\documentclass{amsart}
\usepackage{amsfonts,amssymb,amscd,amsmath,enumerate,verbatim,calc}
\usepackage[all]{xy}

\newcommand{\TFAE}{The following conditions are equivalent:}

\newcommand{\wrt}{with respect to}

\newcommand{\by}{\mathbf{y} }

\newcommand{\n}{\mathfrak{n} }
\newcommand{\m}{\mathfrak{m} }

\newcommand{\Z}{\mathbb{Z} }
\newcommand{\Q}{\mathbb{Q} }

\newcommand{\rt}{\rightarrow}

\newcommand{\ov}{\overline}

\newcommand{\hgt}{\operatorname{height}}
\newcommand{\soc}{\operatorname{soc}}

\newcommand{\Mod}{\operatorname{Mod}}

\newcommand{\chars}{\operatorname{char}}

\newcommand{\Aut}{\operatorname{Aut}}

\newcommand{\Hom}{\operatorname{Hom}}

\newcommand{\frank}{\operatorname{frank}}
\theoremstyle{plain}

\newtheorem{theorem}{Theorem}[section]
\newtheorem{corollary}[theorem]{Corollary}
\newtheorem{lemma}[theorem]{Lemma}

\theoremstyle{definition}

\newtheorem{s}[theorem]{}
\newtheorem{remark}[theorem]{Remark}

\newtheorem{example}[theorem]{Example}

\theoremstyle{remark}
\newtheorem{claim*}{\it Claim:}

\begin{document}
\title{On $G$-invariant Gorenstein ideals}
\author{Tony.~J.~Puthenpurakal}
\date{\today}
\address{Department of Mathematics, IIT Bombay, Powai, Mumbai 400 076}

\email{tputhen@math.iitb.ac.in}
\begin{abstract}
Let $k$ be a field and $G \subseteq Gl_n(k)$ be a finite group with $|G|^{-1} \in k$. Let $G$ act linearly on $A = k[X_1, \ldots, X_n]$ and let $A^G$ be the ring of invariant's. Suppose there does not exist any non-trivial one-dimensional representation of $G$ over $k$. Then  we show that if $Q$ is a $G$-invariant homogeneous ideal of $A$ such that $A/Q$ is a Gorenstein ring then $A^G/Q^G$ is also a Gorenstein ring.
\end{abstract}
\maketitle

\section{introduction}
Let $k$ be a field. Let $A=k[X_1, \ldots, X_n]$ and let $G$ be a finite subgroup of $Gl_n(k)$.
Let $G$ act linearly on $A$. Assume $|G|^{-1} \in k$. \emph{Throughout this paper $G$ will not be the trivial group.}
Also assume $A^G$ the ring of invariant's of $G$ is Gorenstein.
A natural  question is  that if $Q$ is a $G$-invariant homogeneous ideal of $A$ such that
$A/Q$ is a Gorenstein ring then is $A^G/Q^G$ also a Gorenstein ring? Simple examples shows that this is not always the case (see Example \ref{ex}). The objective of this paper is that a simple group theoretic condition ensures this. We prove
\begin{theorem}\label{main-graded}
Let $k$ be a field and $G \subseteq Gl_n(k)$ be a finite group with $|G|^{-1} \in k$. Assume further that there does not exist any non-trivial one-dimensional representation of $G$ over $k$. Let $G$ acts linearly on $A = k[X_1, \ldots, X_n]$ and let $A^G$ be the ring of invariants of $G$. If $Q$ is a homogeneous $G$-invariant ideal in $A$ such that $A/Q$ is a Gorenstein ring then $A^G/Q^G$ is also a Gorenstein ring. Furthermore $a(A^G/Q^G) = a(A/Q)$.
\end{theorem} 
In the above Theorem the numbers $a(A^G/Q^G), a(A/Q)$ denotes the a-invariant of $A^G/Q^G$ and $A/Q$ respectively. 
\begin{remark}\label{intrormk}
1) It is easy that $G \subseteq SL_n(k)$ if $G$ does not have non-trivial one dimensional representations over $k$. So $A^G$ is Gorenstein by a result of Watanabe \cite[Theorem 1]{W}. 

2) The condition that $G$ has no non-trivial one-dimensional representation over $k$ depends both on $G$ and $k$. 
\begin{enumerate}[\rm(a)]
\item If $k=\mathbb{C}$. Then  there does not exist non-trivial group homomorphism
$ \eta: G \to \mathbb{C}^*$ other than the identity if and only if $G = [G, G]$. In particular, $G$ is not solvable. 
\item If $G=[G, G]$ then over any field $k$ there does not exist any one-dimensional representation of $G$ over $k$.
\item However for particular fields the condition $G=[G, G]$ can be relaxed. 

Take $k=\mathbb{R}$. Then if $G$ is any group of odd order ($\geq 3$) then there does not exists any non-trivial one-dimensional representation of $G$ over $\mathbb{R}$. In particular, there exists plenty of solvable groups which satisfy the assumption of our Theorem. See section 4 for some  sufficient condition on $G$ where $k=\mathbb{Z}/p \mathbb{Z}, \mathbb{Q}_p$, finite extensions of $\mathbb{Q}$ which yield an non-existence of non-trivial group homomorphism $G \to k^*$.

Finally we note that we can give a $G$-invariant version of a basic construction in \cite[3.2.11]{BH} to get plenty of $G$-invariant homogeneous ideals $Q$ in $A$ such that $A/Q$ is Gorenstein (for any finite group $G \subseteq Gl_n(k)$). 
\end{enumerate}
\end{remark}

Another natural question is whether an analogue of Theorem \ref{main-graded} holds in the complete case. 
\begin{theorem}\label{complete}
Let $R = k[[X_1,\ldots, X_n]]$. Let $G$ be a finite subgroup of $GL_n(k)$ acting linearly on $R$. Assume $|G|^{-1} \in k$. Further assume that $G$ has no non-trivial one dimensional 
representations over $k$.
Let $R^G$ be the ring of invariants of $G$.  Let $Q$ be a $G$-invariant Gorenstein ideal of $R$. Then $R^G/Q^G$ is also a Gorenstein ring.
\end{theorem}

Although Theorem \ref{complete} implies Theorem \ref{main-graded} we believe the latter is "\textit{more natural}". So we give a detailed proof of Theorem \ref{main-graded}. For Theorem \ref{complete} we mostly sketch the proof. However we give detailed proof of two preliminary  results which do not follow from the graded case.

We now describe in brief the contents of this paper. In section two we discuss some preliminary results that we need. In section three we give a method for constructing $G$-invariant Gorenstein ideals.  In section four we give bountiful 
examples of Groups and fields satisfying our hypothesis. In section five we discuss some consequences for rings of invariant's if the group has no non-trivial one dimensional representations over $k$. The proof of Theorem \ref{main-graded}
is proved by first proving an Artininan analogue which we discuss in section 6.
In section seven we prove Theorem \ref{main-graded}. Finally in section eight we prove Theorem \ref{complete}.

\section{Preliminaries}

In this section we discuss a few preliminaries that we need. Practically all the results in this section are known. However we provide a few proofs as we do not have a convenient reference.

\s(Convention:) If $T$ is a commutative ring and $Q$ is an ideal in $T$, then we say $Q$ is a Gorenstein ideal if $T/Q$ is a Gorenstein ring. \emph{Note we do not 
require that projective dimension of $T/Q$ is finite.}

\s Let $S= \bigoplus_{n \geq 0} S_n$ be a graded ring (not necessarily commutative). Let ${}^*\Mod(S)$ denote the category of all left graded $S$-modules. Let  $M, N$ be graded left $S$-module. 
 For $d \in \Z$, set 
 \[
 {}^*\Hom_d(M, N)= \{f: M \rt N \mid f \text{ is } S\text{-linear and }f(M_i) \subseteq N_{i+d} \ \text{ for all } i \in \Z\}.
 \] 
 Set ${}^*\Hom_S(M, N)= \bigoplus_{i \in \Z} \Hom_d(M, N)$.\\
  In general,
 ${}^*\Hom_S(M, N) \neq \Hom_S(M, N)$ but equality holds if $M$ is finitely presented.

\s\label{tony} Let $R= \bigoplus_{i \geq 0} R_i$ be a standard graded $k=R_0$ algebra. Let
$\theta: G \rt \Aut(R)$ be a group homomorphism where $G$ is a finite group such that 
\begin{enumerate}[\rm (a)]
\item $|G|^{-1} \in k$.
\item For all $g \in G$ we have $gR_i \subseteq R_i$ for all $i \geq 0$.
\item $g(c)=c$ for all $c \in R_0=k$.
\end{enumerate}
Let $R^G$ be the ring of invariant's of $G$. Note we have the {\it Reynolds operator}
\begin{align*}
&\rho: R \rt R^G\\
&r \rt \frac{1}{|G|} \sum_{\sigma \in G} \sigma(r).
\end{align*}

\s Set $R*G$ be the skew-group ring. Recall that \[R*G= \{\sum_{\sigma} a_{\sigma} \sigma ~|~ a_\sigma \in R \text{ for all } \sigma \}.\] with multiplication defined as \[(a_{\sigma} \sigma)(a_{\tau} \tau)= a_\sigma \sigma(a_{\tau})\sigma \tau.\] 

\noindent
(a) Note $R*G$ is graded by defining for a homogeneous $r \in R$, $\deg(r \sigma)= \deg r$. 

\noindent
(b) An graded left $R*G$-module $M = \bigoplus_{i \in \Z}M_n$ is precisely a graded left $R$-module on which $G$ acts such that for all $\sigma \in G$ we have 
\begin{enumerate}[\rm (i)]
\item $\sigma(M_i) \subset M_i$ for all $i \in \Z$. 
\item $\sigma(am)= \sigma(a)\sigma(m)$ for all $a \in A$ and $m \in M$. 
\end{enumerate}

\s Let $M$ be a graded left $R*G$-module. Set $M^G= \{m \in M~|~ \sigma(m)=m \text{ for all } \sigma \in G\}$. Clearly $M^G$ is a graded left $R^G$-module. It can also be easily checked that if $u: M \to N$ is $R*G$-linear then $u(M^G) \subseteq N^G$ and the restriction map $\tilde{u}: M^G \rt N^G$ is $R^G$-linear.

\s Let $M, N$ be graded left $R*G$-modules. Then ${}^*\Hom_R(M, N)$ has a natural structure of a $R*G$-module defined as follows

Let $f \in {}^*\Hom_R(M, N)_c$ and $\sigma \in G$. Set $\sigma f: M \to N$ by $(\sigma f)(m)= \sigma \left(f(\sigma^{-1}m)\right)$. Notice that $\sigma f \in {}^*\Hom_R(M, N)_c$. 

Clearly ${}^*\Hom_R(M, N)^G \cong {}^*\Hom_{R*G}(M, N)$.

\begin{lemma}\label{P5}
Let $M, N$ be graded left $R*G$-modules and $\theta: M \rt N$ be $R*G$-linear ($\theta$ is homogeneous). Let $E$ be a graded left $R*G$-module. Then 
\begin{enumerate}[\rm 1)]
\item ${}^*\Hom_R(E, M) \xrightarrow{\Hom(E, \theta)} {}^*\Hom_R(E, N)$ is $R*G$-linear.
\item ${}^*\Hom_R(N, E) \xrightarrow{\Hom(\theta, E)} {}^*\Hom_R(M, E)$ is $R*G$-linear.
\end{enumerate}
\end{lemma}

\begin{proof}
We prove (2). Proof of (1) is similar.

Set $\psi= \Hom(\theta, E)$. We want to show that $\psi(\sigma f)= \sigma \psi(f)$ for any $\sigma \in G$. 

Now for any $m \in M$ we have $[\psi(\sigma f)](m)= [(\sigma f) \circ \theta](m)= (\sigma f)\left(\theta(m)\right)= \sigma f \left(\sigma^{-1} \theta(m)\right)$.

 Again for any $m \in M$ we have $[\sigma \psi(f)](m)= [\sigma(f \circ \theta)](m)= \sigma [f \circ \theta (\sigma^{-1}m)]= \sigma [f\left(\theta(\sigma^{-1}m)\right)]= \sigma f \left(\sigma^{-1}\theta(m)\right)$ as $\theta$ is $R*G$-linear. Thus $\psi$ is $R*G$-linear.
\end{proof}

\begin{lemma}\label{P6}
\begin{enumerate}[\rm (i)]
\item 
$(R*G)^G= \{ \sum_{\sigma \in G} \sigma(r)\sigma~:~ r \in R \}$.
\item
$(R*G)^G$ has a natural structure of a right $R$-module.
\item 
The map $\psi: R \rt (R*G)^G$ defined by $\psi(r)= \sum_{\sigma \in G} \sigma(r) \sigma$ is a  isomorphism as left $R^G$-module and right $R$-modules.
\end{enumerate}
\end{lemma}

\begin{proof}
(i) Let $u = \sum_{\sigma} a_{\sigma} \sigma \in (R*G)^G$. Then $\tau(u)= \sum_{\sigma} \tau(a_{\sigma})\tau \sigma$. Now $a_{\tau\sigma}= \tau(a_\sigma)$. If $\sigma=1$ then $a_{\tau}=\tau(a_{\bf 1})=\tau(r)$ where $r=a_{\bf 1} \in R$. Thus $u= \sum_{\tau\in G} \tau(r)\tau$.

(ii) The right $R$-action on $(R*G)^G$ is defined as follows. 
Let $u = \sum_{\sigma \in G} \sigma(r)\sigma \in (R*G)^G$ and $s \in R$. Then set 
\begin{align*}
us &= (\sum_{\sigma \in G} \sigma(r)\sigma)s \\
&= \sum_{\sigma \in G} \sigma(r)\sigma(s)\sigma \\
&=\sum_{\sigma \in G} \sigma(rs)\sigma
\end{align*}

(iii) Clearly $\psi$ is surjective by (i). Also if $\sigma(r)=0$ for all $\sigma \in G$ then ${\bf 1}(r)=0$ and hence $r=0$. So $\psi$ is one-one. Let $a \in R^G$. Then $\psi(ar)= \sum_{\sigma}\sigma(ar) \sigma=a \psi(r)$. So $\psi$ is $R^G$-linear as left $R^G$-modules 
 
 Also note that $\psi(rs)= \sum_{\sigma \in G} \sigma(rs)\sigma  = \psi(r)s$
Thus $\psi$ is $R$-linear. So $R \cong (R*G)^G$ as right $R$-modules and left $R^G$-modules. 
\end{proof}

\begin{remark}
By Lemma \ref{P6} it follows that we have an isomorphism of left $R$-modules \[{}^*\Hom_{R^G}((R*G)^G, R^G) \cong {}^*\Hom_{R^G}(R, R^G).\]
\end{remark}

\begin{lemma}
The natural map $\eta: \Hom_{R*G}(R*G, R) \rt \Hom_{R^G}((R*G)^G, R^G)$ is $R$-linear.
\end{lemma}

\begin{proof}
Let $u \in \Hom_{R*G}(R*G, R)$. Note that $ru:R*G \rt R$ is defined by $(ru)(\xi)=u(\xi)r=\xi u(1)r$. 

Also if $\theta \in \Hom_{R^G}((R*G)^G, R^G)$ then $r\theta $ is defined as $(r\theta)(\xi)=\theta(\xi r)$.

 We want to show that $\eta(ru)=r\eta(u)$. Let $\xi \in (R*G)^G$. Then $((\eta(ru)))(\xi)= (ru)(\xi)=\xi u(1)r$. 
 Again $(r \eta(u))(\xi) = \eta(u)(\xi r)=u(\xi r)=\xi r u(1)= \xi u(1)r$ (as $R$ is commutative). Thus $\eta(ru)=r\eta(u)$ and hence $\eta$ is $R$-linear. 
\end{proof}

\s Let $T = \bigoplus_{n \geq 0}T_n$ be a commutative  Noetherian graded $k = S_0$-algebra (not-necessarily standard graded). Let $a(T)$ be the a-invariant of $T$, see \cite[3.6.13]{BH}.

\section{Construction of a G-invariant Gorenstein Ideals}
Let $A=k[X_1, \ldots, X_d]=\bigoplus_{n \geq 0}A_n$ (standard grading). 
Let  $G$ be a finite subgroup of $Gl_n(k)$. Let $G$ act linearly on $A$.
In this section we adapt a construction from \cite[3.2.11]{BH} to show that there are plenty of $G$-invariant ideals $Q$ in $A$ such that $A/Q$ is a Gorenstein ring.

\s  We first recall the construction from \cite[3.2.11]{BH}. 

Let $\phi: A_m \to k$ be a non-trivial linear map. Set $I_0=0$; $I_j=A_j$ for all $j \geq m+1$; $I_m = \ker \phi$ and 
\[I_j =\{a \in A_j ~|~ \phi(a A_{m-j})=0\} \quad \text{for }1 \leq j \leq m-1.\] Then $I = I(\phi) = \bigoplus_{j \geq 0} I_j$ is a proper homogeneous ideal in $A$ such that $A/I$ is Gorenstein.

\s {\bf G-invariant construction:} Suppose $k^{\#}= k $  as vector space and $k^{\#}$ is a $G$-module. Suppose $\phi: A_m \to k^{\#}$ is a non-trivial $G$-invariant map.

\noindent
{\it Claim:} $I= I(\phi)$ is a $G$-invariant ideal.

\noindent
{\it Proof.} Let $a \in I_j$ and $\sigma \in G$. We want to show that $\sigma(a) \in I_j$. If $j \geq m+1$ then $I_j=A_j$. So nothing to show. For $j=m$ we have $I_m= \ker \phi$ which is clearly $G$-invariant as $\phi$ is $G$-linear. Since $I_0=0$, so nothing to show. Now let $1 \leq j \leq m-1$. We want to show $\phi(\sigma(a) A_{m-j})=0$. Now $\sigma: A_{m-j} \to A_{m-j}$ is an isomorphism. Let $\eta \in A_{m-j}$. Then $\eta = \sigma(t)$ for some $t \in A_{m-j}$. Now $\phi(\sigma(a)\eta)= \phi(\sigma(a)\sigma(t))= \phi(\sigma(at))= \sigma(\phi(at))=\sigma(0)=0$. Thus $\phi(\sigma(a)A_{m-j})=0$ implies $\sigma(a) \in I_j$. Thus $I$ is a  $G$-invariant ideal.

By \cite[3.2.11]{BH} we also get $A/I$ is a Gorenstein ideal.

\s {\bf Specific examples:} Let $A^G$ be the ring of invariant's of $G$.  Suppose $A_m^G \neq 0$ and let $\eta: A_m^G \to k$ be any non-trivial map. Give $k$ the  trivial $G$-action. Consider the following diagram.
\[
\xymatrix{
 A_m  \ar[rd]_{\phi} \ar@{->>}[r]^{\rho} & A_m^G \ar[d]^{\eta} \\
  &k }
\]
Clearly $\phi=\eta \circ \rho$ is non-trivial. Now $\phi(\sigma a)= \eta(\rho(\sigma(a)))=\eta(\rho(a))= \phi(a)$. Note $\sigma\phi(a)=\phi(a)$ ( as the action of $G$ on $k$ is trivial).

The following is an example of $G$-invariant ideal $Q$ of $A$ such that $Q^G =Q \cap A^G$ is not a Gorenstein ideal.
\begin{example}\label{ex}
Take $A=k[X, Y]$ and $G= \{ 1, \sigma \}$ with $\chars k \neq 2$. Let $\sigma X =-X$ and $\sigma Y= -Y$. Then it is easy to check that $A^G=k[X^2, XY, Y^2]$. Clearly $A^G$ is Gorenstein. Now $k$ has two distinct structure as $G$-modules. First is trivial. Second is induced by the map $\eta: G \to k^*$ where $\eta(\sigma)=-1$.  Let $ {}^{\#}k$ be $k$ as a vector space with $G$-action defined by $\eta$.  Consider the linear map $\alpha: A_3 \to {}^{\#}k$ where $\alpha(X^3)=1$, $\alpha(X^2Y)=1$, $\alpha(XY^2)=0$ and $\alpha(Y^3)=0$. Clearly $\alpha$ is $G$-invariant (as $\sigma(X^iY^j)= -X^iY^j$ if $i+j$ is odd). Now $I_0 =0$, $I_j= A_j$ for all $j \geq 4$ and $I_3= (X_3-X^2Y, XY^2, Y^3)$.

\noindent
{\it Claim }1. $I_1=0$.

\noindent
{\it Proof.} Since $\phi(X \cdot X^2) \neq 0$ so $X \notin I_1$. Again 
$\phi(Y \cdot X^2) \neq 0$ so $Y \notin I_1$. If $aX+bY \in I_1$ where $a, b \in k^*$, then $0= \phi((aX+bY)X^2)= a \phi(X^3)+b \phi(X^2Y)= a+b$. So $b= -a$. But $\phi((X-Y)\cdot XY)= \phi(X^2Y-XY^2)= \phi(X^2Y)-\phi(XY^2)= 1-0 \neq 0$. So $X-Y \notin I_1$. Thus $I_1=0$. 

\noindent
{\it Claim }2. $I_2=k Y^2$.

\noindent
{\it Proof.} Clearly $Y^2 \in I_2$. Also clearly $X^2, XY \notin I_2$. If $u= aX^2+bXY+cY^2 \in I^2$ with $a, b \in k^*$ and $c \in k$, then $0=\phi(u\cdot X)= a \phi(X^3)+b \phi(X^2Y)= a+b$. So $b=-a$. Take $u= X^2-XY+cY^2$. Then $\phi(uY)= \phi(X^2Y)-\phi(XY^2)+c\phi(Y^3)=1 \neq 0$. So $u \notin I_2$. Hence $I_2=kY^2$. 

Hilbert series of $A/Q$ is $1+2z+2z^2+z^3$. Now $A^G/Q^G= k \oplus 0 \oplus (kX^2 + kYX)$. Thus $(A^G/Q^G)_0= k$, $(A^G/Q^G)_1=0$ and $(A^G/Q^G)_2= kX^2 + kYX$. Clearly $A^G/Q^G$ is not Gorenstein.
\end{example}

The following example shows that even if $Q^G$ is a Gorenstein ideal in $A^G$ it may happen that $a(A^G/Q^G)< a(A/Q)$. 
\begin{example}
Take $A=k[X, Y]$ and $G= \{ 1, \sigma \}$ with $\chars k \neq 2$. Let $\sigma X =-X$ and $\sigma Y= -Y$. Then it is easy to check that $A^G=k[X^2, XY, Y^2]$. Define $\phi: A^3 \to k^{\#}$ by $\phi(X^3)=1$, $\phi(X^2Y)=\phi(XY^2)=\phi(Y^3)=0$. Then it can be checked that $I_0=0$, $I_1=(Y)$, $I_2= ( XY, Y^2)$, $I^3=( X^2Y, XY^2, Y^3)$. Clearly $A/I= k \oplus kX \oplus kX^2 \oplus k X^3$. Now $A^G/I^G= k \oplus k X^2$ is Gorenstein. Note that $a(A^G/I^G)=2 < 3 = a(A/I)$.
\end{example}

\section{Examples of Groups and Fields Satisfying our Condition}
In this section we give examples of finite groups $G$ and fields $k$ such that $G$ does not have no non-trivial one dimensional representations over $k$. Throughout we assume that $|G|^{-1} \in k$. The results of this section are certainly  known
to workers in representation theory of groups. However we state it here to show the abundance of groups and fields satisfying our condition.

We begin by an easy result
\begin{lemma}
\TFAE
\begin{enumerate}[\rm 1)]
\item There does not exist any one-dimensional representation of $G$ over $k$.
\item If $\eta: G \to k^*$ is a group homomorphism, then $\eta(g)=1$ for all $g \in G$.
\item Either 
\begin{enumerate}[\rm (a)]
\item $G= [G;G]$,
\item If $r=|G| /[G;G]| \geq 2$ and $p|r$ is a prime then there does not exist $p^{th}$ primitive root of unity in $k$.
\end{enumerate}
\end{enumerate}
\end{lemma}

\begin{proof}
1) and 2)are clearly equivalent. 

3)b) Suppose $p \mid r$ and there exists a $p^{th}$ root of unity in $k$. Then $G/[G, G]= \Z/p^r\Z \oplus H$ where $H$ is a finite abelian group. Now we have $\Z/p^r\Z \twoheadrightarrow \Z/p\Z \hookrightarrow k^*$. Let $\eta$ be the composite map
\[
\xymatrix{
&G \ar@{->>}[r]& G/[G, G] \ar@{->>}[d] \\
 & &\Z/p^r\Z \ar@{->>}[r] &\Z/p\Z \ar@{^{(}->}[r]& k^* }
\]
So $\eta: G \to k^*$ is non-trivial.

Conversely, if $\eta: G \to k^*$ is non-trivial, then $\eta: G/[G, G] \to k^*$ is also non-trivial. From here it is easy to show that there exists $p \mid |G/[G, G]|$ such that there exists a primitive $p^{th}$ root of unity in $k$.
\end{proof}
In this table we collect examples of Groups and fields satisfying our condition.
\emph{Throughout we assume $G$ is not the trivial group and $p,q$ are prime numbers. Also $\mathbb{Q}_q$ denotes the field of $q$-adic numbers}

\begin{center}
\begin{tabular}{ |p{4.5cm}|p{7.5cm}| }
\hline
{\bf $k$} & { \bf Sufficient condition for $G$ not to have one-dimensional representation over $k$}\\
\hline\hline
$k=\mathbb{C}$ & $G=[G:G]$\\
\hline
$k=\mathbb{R}$ &  $G=[G;G]$ or $|G|$ odd\\
\hline
$k=\Z/2\Z$ & $G=[G,G]$ or $|G|$ is odd \\
\hline
$\chars k =p>0$ and $k = \ov{k}$
 & $G=[G:G]$\\
\hline
$k=\Z/q\Z$  & $G=[G:G]$ or if  $p\mid |G|$ then $p \nmid q-1$\\
\hline
$k=$  & $G=[G:G]$ or if  $p\mid |G|$ then $p \nmid q-1$ \\
\hline
$ K$ is a finite extension of $\mathbb{Q}$
 & If  $p \mid |G|$ then  $p > \dim_\Q K +1$. \\
\hline
\end{tabular}
\end{center}

\section{Some Consequences for Ring of Invariant's of Groups having no non-trivial one-dimensional Representation}
\s \label{hypoth-conseq}In this section we assume that $G$ has no non-trivial one-dimensional representation over $k$ and examine it's consequences for ring of invariant's. 
Throughout we assume that the setup is as in \ref{tony}.  We will also assume that both $R$ and $R^G$ are Gorenstein.

\begin{lemma}\label{ginvar-gen}
Let $M= \bigoplus_{n \geq 0}M_n$ be a graded $R*G$-module such that $M \cong R/J$ as $R$-modules. Then if $u \in M_0$ is a generator of $M$ then $u\in M^G$. So $J$ is $G$-invariant and $M \cong R/J$ as $R*G$-modules. 
\end{lemma}
\begin{proof}
For $\sigma \in G$ we know that $\sigma(u)$ is a also a generator of $M$. Now $\sigma(u)= a_{\sigma}u$ where $a_\sigma \in k^*$. Let $\tau \in G$. Then $(\tau \sigma)(u)= \tau(a_\sigma)\tau(u)=a_\sigma a_\tau u$ (as $\tau(a_\sigma)=a_\sigma$). But $(\tau\sigma)(u)=a_{\tau \sigma}u$. Thus $a_{\tau \sigma}=a_\sigma a_\tau= a_\tau a_\sigma$. So we have a group homomorphism $\eta: G \to k^*$ where $\eta(\sigma)= a_\sigma$. But by our assumption $\eta \equiv 1$. So $a_\sigma =1$ for all $\sigma \in G$ and hence $\sigma(u)=u$ for all $\sigma \in G$.  Therefore the map $\phi: R \to M$ defined by $\phi(r)= ru$ is $G$-invariant. So $\ker \phi=I$ is $G$-invariant. Thus $M \cong R/I$ as $R*G$-modules.
\end{proof}

\begin{lemma}\label{ainv}(with hypotheses as in \ref{hypoth-conseq})
$a(R) = a(R^G)$.
\end{lemma}
\begin{proof}
Let $\underline{\bf x}= x_1, \ldots, x_d$ be a homogeneous system of parameters of $R^G$. Then $(R/\underline{\bf x}R)^G=R^G/\underline{\bf x}R^G$. It is suffices to show that $a(R/\underline{\bf x}R) = a(R^G/\underline{\bf x}R^G)$, see \cite[3.6.14]{BH}. So we can assume that $\dim R = \dim R^G=0$. 

Let $R=k \oplus R_1 \oplus R_2 \oplus \cdots \oplus R_s$ where $R_s = \soc(R) = ku \neq 0$. Then $B= R^G = k \oplus B_1 \oplus B_2 \oplus \cdots \oplus B_s$ where
 $B_i =R^G \cap R_i$. Thus it is suffices to show that $u \in R^G$.
 
 Let $\m$ be the maximal ideal of $R$. As $R$ is Gorenstein we have
  $R_s = \soc(R) (0: \m)$. Note for $\sigma \in G$ we have $\m \sigma(u)=  \sigma(\m)\sigma(u )=\sigma(\m u)=\sigma(0)=0$. So $\sigma(u)$ also generates $\soc(R) =R_s$. Thus $\sigma(u)= a_{\sigma}u$ where $a_\sigma \in k$. Let $\tau \in G$. Then $(\tau \sigma)(u)= \tau(a_\sigma)\tau(u)=a_\sigma a_\tau u$ (as $\tau(a_\sigma)=a_\sigma$). But $(\tau\sigma)(u)=a_{\tau \sigma}u$. Thus $a_{\tau \sigma}=a_\sigma a_\tau= a_\tau a_\sigma$. Now $\eta: G \to k^*$ defined by $\eta(\sigma)= a_\sigma$ is a group homomorphism. Now by our assumption $\eta \equiv 1$. So $a_\sigma =1$ for all $\sigma \in G$ and hence $\sigma(u)=u$ for all $\sigma \in G$. Therefore $u \in R^G$ and we are done.
\end{proof}
\begin{remark}
Lemma \ref{ainv} need not hold if $G$ has a non-trivial one dimensional representation over $k$; see Example 3.5.
\end{remark}
\begin{corollary}\label{rho-gen}(with hypotheses as in \ref{hypoth-conseq})
Let $\m$ be the maximal ideal of $R$.
Let $\rho: R\to R^G$ be the Reynolds operator. Then $\rho \notin \m {}^*\Hom_{R^G}(R, R^G)$. 
\end{corollary}

\begin{proof}
Let $a(R)= a(R^G)= s$. Then ${}^*\Hom_{R^G}(R, R^G(s))=R(s)$ (as $R$ and $R^G$ are Gorenstein). So ${}^*\Hom_{R^G}(R, R^G)= R$ as graded $R$-modules. Note that $\rho \in \Hom_{R^G}(R, R^G)_0$ is non-zero. The result follows.
\end{proof}

The following result will be used in the proof of Theorem \ref{complete}.
\begin{lemma}\label{frank}
(with hypotheses as in \ref{hypoth-conseq})
$\frank_{R^G}(R) = 1$.
\end{lemma}
\begin{proof}
Suppose if possible $\frank_{R^G} R \geq 2$.
 Note the inclusion $R^G \rt R$ splits. Say $R = R^G \oplus L$. It follows that $L$ has a free summand.
 Also note that $L \subseteq R_+$. So $L = R^G(-c) \oplus V$ for some $c > 0$ and some graded $R^G$-module $V$.
 
 By Lemma \ref{ainv} we get $a(R) = a(R^G)$. As they are both Gorenstein we have $ { }^*\Hom_{R^G}(R, R^G) \cong R$.
 However we also have $ { }^*\Hom_{R^G}(R, R^G)  = R^G \oplus R^G(+c) \oplus { }^*\Hom_{R^G}(V, R^G)$. 
 So we get $R_{-c} \neq 0$ which is a contradiction.
\end{proof}

\section{The Artin Case}
\s \label{artin-gor} Throughout we assume that the setup is as in \ref{tony}.  We will also assume that both $R$ and $R^G$ are Gorenstein. We will assume that $G$ has no non-trivial one-dimensional representation over $k$. Furthermore we also assume that 
$R$ (and hence $R^G$) is Artinian.

The main result of this section is
\begin{theorem}\label{main-artin}
(with hypotheses as in \ref{artin-gor}) Let $Q$ be a homogeneous ideal in $R$ with $R/Q$ is Gorenstein, then $R^G/Q^G$ is also a Gorenstein ring. Furthermore 
$a(R/Q) = a(R^G/Q^G)$.
\end{theorem}

To prove the main result of this section we need the following result.

\begin{lemma}\label{iso}
For any finitely generated $R*G$-module $M$ the map

\[\eta_M: {}^*\Hom_{R*G}(M, R) \rt  {}^*\Hom_{R^G}(M^G, R^G)\] defined by $\eta_M(u)= u|_{M^G}$ is an isomorphism of $R^G$-modules.
\end{lemma}

\begin{proof}
We first show \[\eta_{R*G}: {}^*\Hom_{R*G}(R*G, R) \longrightarrow  {}^*\Hom_{R^G}((R*G)^G, R^G)\] is an isomorphism. 
We note that both $R*G$ and $(R*G)^G$ have natural structures as right $R$-modules. By Lemma 2.9 we get that $\eta_{R*G}$ is in-fact $R$-linear.

Note that we have an isomorphism of $R$-modules 
\begin{align*}
{}^*\Hom_{R*G}(R*G, R) &\longrightarrow  R\\
f &\longrightarrow f(1)
\end{align*}
We can define the inverse map $\phi: R \rt {}^*\Hom_{R*G}(R*G, R)$ by $\phi(r)=f_r$ where $f_r(s)=sr$ for all $s \in R*G$.

Thus we get maps of $R$-modules 
\[R \overset{\Phi}{\rt} {}^*\Hom_{R*G}(R*G, R) \overset{\eta}{\rt}{}^*\Hom_{R^G}((R*G)^G, R^G) \simeq {}^*\Hom_{R^G}(R, R^G) \simeq R.\] Notice that $\psi: R \to {}^*\Hom_{R^G}(R, R^G)$ defined by $\psi(1)= \rho$ is an isomorphism (as 
$a(R)=a(R^G)$; see \ref{rho-gen}. Let $\xi= \sum_\sigma \sigma(r)\sigma \in (R*G)^G$. Then  $[(\eta \circ \Phi)(1)](\xi)= \eta(\Phi(1))\left(\sum \sigma(r)\sigma\right)= \left(\sum_\sigma \sigma(r)\sigma\right)(1)= \sum_\sigma \sigma(r)=|G| \rho(r)$. Thus $\eta \circ \Phi=|G| \rho$ is a generator of ${}^*\Hom_{R^G}((R*G)^G, R^G)$. Since $\eta \circ \Phi$ is $R$-linear so we get $\eta \circ \Phi$ is an isomorphism. Hence $\eta$ is an isomorphism.

It is easy to verify that if $M, N$ are graded $R*G$-modules and $\alpha: M \to N$ is $R*G$-linear then we have a commutative diagram of $R^G$-modules

\begin{align*}
\xymatrixrowsep{2.5pc} \xymatrixcolsep{5pc}
\xymatrix{
 {}^*\Hom_{R*G}(N, R) \ar[r]^{\Hom(\alpha, R)} \ar[d]^{\eta_N}&  {}^*\Hom_{R*G}(N, R)\ar[d]^{\eta_M}\\
{}^*\Hom_{R^G}(N^G, R^G) \ar[r]^{\Hom(\alpha^G, R^G)}& {}^*\Hom_{R^G}(M^G, R^G)}
\end{align*}
Furthermore, as $\eta_{R*G}: {}^*\Hom_{R*G}(R*G, R) \longrightarrow  {}^*\Hom_{R^G}((R*G)^G, R^G)$ is an isomorphism we get $\eta_F$ is also isomorphism for any graded finite free $R*G$-module $F$. 

We now prove $\eta_M$ is an isomorphism for any finitely generated graded $R*G$-module $M$. We have finite presentation of $M$ (as $R*G$ is a Noetherian ring) \[F \overset{\theta}{\rt} L \overset{\epsilon}{\rt} M \rt 0\] as graded $R*G$-module with $F, L$ finite free graded $R*G$-modules. Consider the following diagram,
\[
  \xymatrix
{
 0
 \ar@{->}[r]
  & {}^*\Hom_{R*G}(M, R)
    \ar@{->}[d]^{\eta_M}
\ar@{->}[r]
 & {}^*\Hom_{R*G}(L, R)
    \ar@{->}[d]^{\eta_L}
\ar@{->}[r]
&  {}^*\Hom_{R*G}(F, R)
    \ar@{->}[d]^{\eta_F}
 \\
 0
 \ar@{->}[r]
  & {}^*\Hom_{R^G}(M^G, R^G)
  \ar@{->}[r]
 & {}^*\Hom_{R^G}(L^G, R^G)
\ar@{->}[r]
& {}^*\Hom_{R^G}(F^G, R^G)
}
\]

Since $\eta_L$ and $\eta_F$ are isomorphisms so we get $\eta_M$ is also an isomorphism.
\end{proof}

We now give
\begin{proof}[Proof of Theorem \ref{main-artin}]
$R/Q$ is an $R*G$-module. So ${}^*\Hom_R(R/Q, R)$ is also an $R*G$-module. 
But ${}^*\Hom_R(R/Q, R) \cong R/Q(s)$ for some $s \in \Z$ as $R$-modules.
By Lemma \ref{ginvar-gen},  ${}^*\Hom_R(R/Q, R) \cong R/Q(s)$ as $R*G$-modules.
So \\ ${}^*\Hom_R(R/Q, R)^G \cong R^G/Q^G(s)$ as $R*G$-modules.
Now by Lemma \ref{iso} we have the isomorphism of
$R^G$-modules 
\[R^G/Q^G(s) \cong {}^*\Hom_R(R/Q, R)^G \cong {}^*\Hom_{R*G}(R/Q, R) \cong {}^*\Hom_{R^G}(R^G/Q^G, R^G).\] 
As $R^G$ is Gorenstein and $\dim R^G/Q^G= \dim R^G=0$. So ${}^*\Hom_{R^G}(R^G/Q^G, R^G) \cong \omega_{R^G/Q^G}$ up-to shift. Thus it follows that $R^G/Q^G$ is a Gorenstein ring. The assertion $a(R^G/Q^G) = a(R/Q)$ follows from Lemma \ref{ainv} (as $(R/Q)^G \cong R^G/Q^G$).
\end{proof}
\section{Proof of Theorem \ref{main-graded}}

In this section we give:
\begin{proof}[Proof of Theorem \ref{main-graded}]
We note that $G$ acts on $A/Q$. So $A^G/Q^G$ is a Cohen-Macaulay ring. 
Notice by Remark \ref{intrormk}(1) we have $A^G$ is a Gorenstein ring.
Notice that $\hgt Q= \hgt Q^G = \text{(say equal to)} \ g$.
Let $\underline{\bf u}= u_1, u_2, \ldots, u_g \in Q^G$ be an $A^G$-linear sequence.
Then $\underline{\bf u}$ is also an $A$-linear sequence. Set $S = A/(\underline{\bf u})A$. Then 
$S^G=A^G/(\underline{\bf u})A^G$. Now $Q/(\underline{\bf u})A$ is a Gorenstein ideal in $S$ and $\left(Q/(u)A \right)^G= Q^G/(\underline{\bf u})A^G$. Thus it is suffices to prove for $S$ and when $\hgt(\overline{Q})= \hgt(\overline{Q^G}) = 0$. Also note that $S^G$ is a Gorenstein ring.

Let $\underline{\bf y}= y_1, y_2, \ldots, y_r \in S^G$ be a homogeneous system of parameter. Note that $\underline{\bf y}$ is also a homogeneous system of parameter of $S$. Clearly $S/\underline{Q}$ is a maximal Cohen-Macaulay $S$-module (as $\dim S/\overline{Q}= \dim S-\hgt \overline{Q}= \dim S$). So $\underline{\bf y}$ is a $S/\overline{Q}$-regular sequence. Note that $S/(\overline{Q}, \underline{\bf y})= S/(\underline{\bf y})/(\overline{Q}, \underline{\bf y})/(\underline{\bf y})$ is  Gorenstein. 

\noindent
{\it Claim}: $(\overline{Q}, \underline{\bf y}S)^G= (\overline{Q}^G, \underline{\bf y}S^G)$. 

\noindent
{\it Proof.} Clearly $ (\overline{Q}^G, \underline{\bf y}S^G)\subseteq (\overline{Q}, \underline{\bf y}S)^G$. Let $\xi \in (\overline{Q}, \underline{\bf y}S)^G$. Then $\xi= \alpha+ \sum_{i=1}^s y_is_i$. Thus $\xi= \rho(\xi)= \rho(\alpha)+ \sum_{i=1}^s y_i \rho(s_i) \in (\overline{Q}^G, \underline{\bf y}S^G)$.

Also note that $S^G/\overline{Q^G}$ is a maximal Cohen-Macaulay $S^G$-module. So $\underline{\bf y}= y_1, \ldots, y_r$ is a $S^G/\overline{Q^G}$-linear sequence. Also $S^G/\overline{Q^G}$ is a Gorenstein ring if and only if $S^G/(\overline{Q^G}, \underline{\bf y}S^G)$ is a Gorenstein ring. Set $R=S/(\underline{\bf y})S$ and $q= \overline{Q}+(\underline{\bf y})S/(\underline{\bf y})S$ is a $G$-invariant Gorenstein ideal in $R$. $R$ is a standard graded Artin ring. Clearly $R^G= S^G/(\underline{\bf y})S^G$ and $q^G= \overline{Q^G}+(\underline{\bf y})S^G/(\underline{\bf y})S^G$. By Theorem \ref{main-artin} we have $R^G/q^G$ is a Gorenstein ring and that $a(R^G/q^G) = a(R/q)$. But $R^G/q^G= S^G/ \overline{Q^G}+(\underline{\bf y})S^G$. So $S^G/\overline{Q^G}$ is a Gorenstein ring. Thus $A^G/Q^G=S^G/\overline{Q^G}$ is a Gorenstein ring. Furthermore a routine calculation also yields $a(A^G/Q^G) = a(A/Q)$; see \cite[3.6.14]{BH}.
\end{proof}

\section{Complete case}

Throughout this section the setup is as in Theorem \ref{complete}.

\s \label{basis-setup} The following fact is well-known. Set $A = k[X_1,\ldots, X_n]$ and let $A^G$ be the ring of invariants of $A$ \wrt \ $G$.
Let $\m = (X_1, \ldots, X_n)$ and set $\m^G = \m \cap A^G$. Then $R^G$ is the completion of $A^G$ \wrt \ $\m^G$.
By Lemma \ref{frank} we get $\frank_{A^G} A = 1$. It is then easy to verify that $\frank_{R^G} R = 1$. Let 
$\by = y_1, \ldots, y_n$ be a system of parameters in $R^G$. Set $S = R/(\by)R$. Note $S^G = R^G/(\by)R^G$.
It is also easy to check that $\frank_{S^G} S = 1$. Also note that $S, S^G$ are Gorenstein rings.

The following two results do not follow from the graded case.
\begin{lemma}\label{n-follow}
 (with hypothesis as in \ref{basis-setup}). Let $\n$ be the maximal ideal of $S$ and let $Q$ be a $G$-invariant ideal of $S$.
 Set $\rho \colon S \rt S^G$  to be the Reynolds operator. We then have 
 \begin{enumerate}[\rm (1)]
  \item
  $\rho \notin \n \Hom_{S^G}(S, S^G)$.
  \item
  If $\Hom_S(S/Q, S) \cong S/Q$ as $S$-modules then it also isomorphic as $S*G$-modules.
 \end{enumerate}
 \end{lemma}
\begin{proof}
(1) Suppose if possible $\rho \in\n \Hom_{S^G}(S, S^G)$.  Say
\[
 \rho = \sum_{i = 1}^{s}a_i \phi_i \quad \text{for some} \ a_i \in \n \ \text{and} \ \phi_i \in \Hom_{S^G}(S, S^G).
\]
We get
\[
 1 = \rho(1) = \sum_{i = 1}^{s}\phi_i(a_i).
\]
It follows that for some $j$ we have $\phi_j(a_j) = v$ is a unit in $S^G$. If $S = S^G \oplus L$ is a splitting of $S$ as $S^G$-modules
via $\rho$ then notice $\n = \n^g \oplus L$.
Say $a_j = b + l$ where $b \in \n^G$ and $l \in L$. Then
$v = \phi_j(a_j) = b\phi_j(1) + \phi_j(l)$. It follows that $\phi_j(l)$ is a unit in $S^G$. Thus $L$ has $S^G$ as a free summand.
So $\frank_{S^G} S \geq 2$. This is a contradiction, see \ref{basis-setup}.

(2) Let $u$ be a generator of $W = \Hom_S(S/Q, S)$. We note that $W$ and $\n W$ are $S*G$-modules.
So $W/\n W$ is also a $S*G$-module. But $W/\n W \cong k$ as $k$-vector spaces. By our assumption on $G$ it follows that
$W/ \n W$ is the trivial $G$-module $k$. It follows that for $\sigma \in G$ we have
$\sigma(u) = u + r_\sigma u$ for some $r_\sigma \in \n$. Therefore
\[
 \rho(u) = u + ru \quad \text{for some} \ r \in \n.
\]
Thus $\rho(u) = (1 + r ) u$. It follows that $\rho(u) \in W^G$ is also a generator of $W$. The result follows. 
\end{proof}
We now give
\begin{proof}(Sketch of a proof of Theorem \ref{complete})
As in the proof of Theorem \ref{main-graded}, after going mod a suitable system of parameters of $R^G$ we may assume
(as in \ref{basis-setup}) that $S, S^G$ are Artin Gorenstein rings and $Q$ is a $G$-invariant ideal in $S$ with $S/Q$ Gorenstein.

As $S$ is Artinian we have $\Hom_S(S/Q, S) \cong S/Q$ as $S$-modules. By Lemma \ref{n-follow}(2) we also have
$\Hom_S(S/Q, S) \cong S/Q$ as $S*G$-modules. As in the proof of Lemma \ref{iso} we get that for any $S*G$ module $M$ we have 
a natural isomorphism of $S^G$-modules
\[
 \eta_M \colon \Hom_{S*G}(M, S) \rt \Hom_{S^G}(M^G, S^G)
\]
(The proof in Lemma \ref{iso} essentially uses the fact that $\rho \notin \n\Hom_{S^G}(S, S^G)$ and that both $S, S^G$ are
Gorenstein).
Thus as $S/Q \cong \Hom_S(S/Q, S)$ as $S*G$-modules we have the following isomorphisms as $S^G$-modules:
\[
 S^G/Q^G  \cong \Hom_{S}(S/Q, S)^G \cong \Hom_{S*G}(S/Q, S) \cong \Hom_{S^G}(S^G/Q^G, S^G).
\]
As $S^G$ is an Artinian Gorenstein ring the result follows.

\end{proof}

\section*{Acknowledgements}
I thank Sudeshna Roy for help in typing this paper.
I also thank Prof. K. Watanabe for some helpful discussions.

\end{document}